\documentclass{amsart}
\usepackage{biblatex} 
\usepackage{graphicx} 
\usepackage{amsmath}
\usepackage{bbm}
\usepackage{amssymb}
\usepackage{amsthm}
\usepackage{hyperref}
\usepackage{tikz}
\usetikzlibrary{tikzmark, fit, calc}
\usepackage{chemfig}
\usepackage[version=3]{mhchem}
\usepackage{qtree}
\usepackage{tabularx}
\usepackage{enumerate}
\usepackage{breqn}
\usepackage[makeroom]{cancel}
\usepackage{mathrsfs}
\usepackage{etoolbox}
\usepackage{caption}
\usepackage{adjustbox}
\newtheorem{thm}{Theorem}[section]
\newtheorem{dfn}[thm]{Definition}
\newtheorem{crl}[thm]{Corollary}
\newtheorem{prp}[thm]{Proposition}
\newtheorem{lm}[thm]{Lemma}
\newtheorem{rmk}[thm]{Remark}

\tikzset{
    main_node/.style={circle,fill=white,draw,minimum size=0.25cm,inner sep=0pt},
    inv_node/.style={circle, fill=white, minimum size=0cm, inner sep=0pt},
    node_2/.style={circle,fill=white,draw,minimum size=0.5cm,inner sep=0pt},
            }
\tikzset {
    treenode/.style = {align=center, inner sep=1pt},
    node_black/.style={treenode, circle, black, draw=black, fill=black},
    node_white/.style={treenode, circle, white, draw=black, fill=white},
    node_black_lb/.style={treenode, circle, white, draw=black, fill=black},
    node_white_lb/.style={treenode, circle, black, draw=black, fill=white},
    node_invisible/.style={treenode, circle, black, draw=white, fill=white}
}

\title{Enumeration of multipartite series-reduced trees}
\author{Medet Jumadildayev}
\address{Institute of Mathematics and Mathematical Modeling, Almaty, Kazakhstan}
\address{Nazarbayev University, Astana, Kazakhstan}
\email{medet.jumadildayev@nu.edu.kz}
\subjclass[2020]{05C05, 05C07, 05C15, 11B73}

\bibliography{references} 

\begin{document}

\begin{abstract}
    We obtain a generating function for the degree sequences and colors of rooted multipartite labeled series-reduced trees. As an application of this result, we determine the number of symbolic ultrametrics (introduced by Böcker and Dress) and increasingly labeled processes. We also find that the number of multipartite labeled series-reduced trees and the colored chain-increasing binary trees are the same. We obtain the number of rooted multipartite unlabeled series-reduced trees. We also find a refinement of the result of Riordan and Shannon.
\end{abstract}

\maketitle

\section{Introduction}

A \textit{rooted tree} is a connected graph without cycles with a distinguished vertex called the \textit{root}. The out-degree of a vertex $v$ in a rooted tree $T$ is the number of children of $v$ in $T$. The \textit{leaves} of a rooted tree $T$ are vertices with out-degree $0$, and all other vertices are called \textit{inner vertices}. We call a rooted tree \textit{series-reduced} if there are no vertices with out-degree $1$.  A rooted series-reduced tree is \textit{leaf-labeled} if one can assign distinct labels to its leaves. For simplicity, we will refer to leaf-labeled rooted series-reduced trees as \textit{labeled series-reduced trees} \cite{RiordanSchroeder}, \cite{RiordanShannon}. A multipartite (or $m$\nobreakdash-partite) labeled series-reduced tree is a tree with inner vertices colored in $m$ colors such that no two vertices have the same color. 

The main result of our paper is the generating function for the degree sequences and colors of the rooted $m\text{-partite}$ labeled series-reduced trees.

\begin{thm} \label{main_result}
Let $m$ denote the number of colors. Additionally, for all colors $1 \leq c \leq m$ denote degree functions $x_c(t)$ and their compositional inverses $x_c^{\langle -1 \rangle}(t)$ as

    \[
    x_c(t) = t + \sum_{n \geq 2} x_{c, n} \dfrac{t ^ n}{n!},
    \]

    \[
    x_c^{\langle -1 \rangle}(t) = t + \sum_{n \geq 2} \dfrac{t ^ n}{n!} x_{c,n} ^ {\langle -1 \rangle}.
    \] 
    Then, the generating function for multipartite rooted labeled series-reduced trees on $m$ colors can be stated in terms of a compositional inverse of the sum of compositional inverses of degree functions as

    \[
    P(m, t,  x) = \left(t +  \sum_{c = 1} ^ {m} \left(x_c ^ {\langle -1 \rangle} (t) - t\right) \right) ^ {\langle -1 \rangle}.
    \]
\end{thm}

Alibek Adilzhan independently discovered a formula for $P(2, t, x)$. A corollary of Theorem \ref{main_result} is a formula for the number of rooted $m\text{-partite}$ labeled series-reduced trees. It is also the number of symbolic ultrametrics in \cite{Dress}, when $M$ is finite.

\begin{crl}\label{main_crl}
    The generating function for the number of symbolic ultrametrics $D: X \times X \rightarrow M$ where $|M| = m$, or equivalently the number of $m\text{-partite}$ labeled series-reduced trees, equals the following expression
    
    \[
    A(m, t) = (t(1 - m) + m\log(1 + t)) ^ {\langle -1 \rangle}.
    \]
    Let $!n_k$ denote the number of derangements of length $n$ which have $k$ cycles. Then, the number $a_s(m)$ of symbolic ultrametrics $D: X \times X \rightarrow M$ such that $|X| = s$ and $|M| = m$ can be computed as

    \[
    a_s(m) = (-1) ^ {s - 1} \sum_{k = 0} ^ {s} (-m) ^ k (!(s + k - 1)_k).
    \]
\end{crl}

\begin{rmk}
    Note that the coefficients of the polynomial $a_s(m)$ are given explicitly.
\end{rmk}

\noindent{\bf Example.} For the table of values $a_s(m)$, please refer to Table \ref{tab:symb_values}. Below we provide polynomials $a_s(m)$ for $s = 1, 2, \cdots, 7$. 

\[
a_1(m) = 1,
\]
\[
a_2(m) = m,
\]
\[
a_3(m) = 3m ^ 2 - 2m,
\]
\[
a_4(m) = 15 m^3-20 m^2+6 m,
\]
\[
a_5(m) = 105 m^4-210 m^3+130 m^2-24 m,
\]
\[
a_6(m) = 945 m^5-2520 m^4+2380 m^3-924 m^2+120 m,
\]
\[
a_7(m) = 10395 m^6-34650 m^5+44100 m^4-26432 m^3+7308 m^2-720 m.
\]

\section{Bell Groups}

In this section, we outline a technique used to extract coefficients from generating functions. We define the notion of Bell Groups and outline their properties. First, let's define Bell Polynomials.

\begin{dfn} A Bell polynomial $B_{n, k}(x) = B_{n, k}(x_1, x_2, \cdots, x_{n - k + 1})$ is defined as
$$B_{n, k}(x) = \sum_{\alpha} \frac{n!}{\alpha_1! \alpha_2! \cdots \alpha_{n - k + 1}!} \left(\frac{x_1}{1!}\right) ^ {\alpha_1} \left(\frac{x_2}{2!}\right) ^ {\alpha_2} \cdots \left(\frac{x_{n - k + 1}}{(n - k + 1)!}\right) ^ {\alpha_{n - k + 1}},$$
where the sum is taken over all sequences $\alpha$ of non-negative integers satisfying following two conditions.

\begin{enumerate}
    \item $\alpha_1 + \alpha_2 + \alpha_3 + \cdots + \alpha_{n - k + 1} = k$,
    \item $\alpha_1 +2 \alpha_2 + 3\alpha_3 + \cdots + (n - k+ 1) \alpha_{n - k + 1} = n$.
\end{enumerate}
\end{dfn}
Let $V$ be a set of infinite sequences ${ v}=(v_1,v_2,\ldots )$ such that $v_1 \neq 0$.
\begin{dfn} Coordinates of a Bell product $ x \circ  y\in  V$ of two sequences $ x, y \in V$ are defined as a sum of Bell polynomials.
    $$( x \circ  y)_n=\sum_{k = 1} ^ {n} x_k B_{n, k} ( y).$$
\end{dfn}
For convenience, we will denote the Bell product as
$$B_n( x,  y) = \sum_{k = 1} ^ {n} x_k B_{n, k}( y).$$
The set $ V$ equipped with the Bell Product is a group. This fact is well known. For completeness, we give a formal proof. 

\begin{thm}\label{bell_group}
    Set $ V$ equipped with the Bell product forms a group. We call it a Bell group.
\end{thm}

\begin{proof}
Our approach involves enumerating non-colored, rooted trees with labeled leaves. All the leaves in these trees are exactly at a distance $2$ from the root. In other words, the parents of the leaves of a tree are the children of the root. We will show that the enumeration of such trees is equivalent to the Bell product operation and also to function composition.

Let us define the weight of a tree $T$ with $n$ leaves as follows
    $$w(T) = \frac{t ^ n}{n!} x_k \prod_{u \in children(root)} y_{|children(u)|},$$
 where $n$ is the number of leaves, $k$ is the number of children of the root, $root$ is the root of the tree $T$, $children(u)$ is the set of children of vertex $u$ in $T$.
        
    Figure 1 illustrates a tree with $9$ leaves, a root with $3$ children, and intermediate vertices with degrees $2, 3, 4$, respectively. That is why, by definition, its weight is $\frac{t ^ 9}{9!}x_3 y_2y_3y_4$. 
    
    \begin{figure}[ht]
        \centering
        \begin{tikzpicture}[grow'=up, level distance=1cm,
        level 1/.style={sibling distance=3cm},
        level 2/.style={sibling distance=1cm}]

        \node[circle,draw,scale=0.7] {}
            child {node[circle,draw,scale=0.7] {}
            child {node[circle,draw,scale=0.7] {1}}
            child {node[circle,draw,scale=0.7] {2}}
        }
        child {node[circle,draw,scale=0.7] {}
        child {node[circle,draw,scale=0.7] {3}}
        child {node[circle,draw,scale=0.7] {4}}
        child {node[circle,draw,scale=0.7] {5}}
        }
        child {node[circle,draw,scale=0.7] {}
        child {node[circle,draw,scale=0.7] {6}}
        child {node[circle,draw,scale=0.7] {7}}
        child {node[circle,draw,scale=0.7] {8}}
        child {node[circle,draw,scale=0.7] {9}}
        };

        \end{tikzpicture}
        \caption{A tree with $9$ leaves and weight $\frac{t ^ 9}{9!}x_3 y_2 y_3 y_4.$}
    \end{figure}
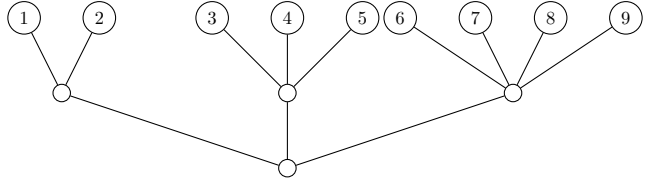
    The generating function of rooted trees with $n$ leaves is
    $$\sum_{|leaves(T)| = n} w(T) = \frac{t ^ n}{n!}\sum_{k = 1} ^ {n} x_k B_{n, k} ( y).$$
    Then, the generating function $A(t)$ of all trees can be expressed as
    \begin{dmath*}
        A(t) = \sum_{n \geq 1} \sum_{|leaves(T)| = n} w(T) =  \sum_{n \geq 1} \frac{t ^ n}{n!} \sum_{k = 1} ^ {n} x_k B_{n, k} ( y).
    \end{dmath*}
    On the other hand, if $y(t) = \sum_{n \geq 1} y_n \frac{t ^ n}{n!}$ is a generating function for subtrees, the generating function $A(t)$ of trees satisfies the following relation
    $$A(t) = \sum_{k \geq 1} x_k \frac{y(t) ^ n}{n!}.$$
    Hence,
    $$\sum_{k \geq 1} x_k \frac{y(t) ^ n}{n!} = \sum_{n \geq 1} \frac{t ^ n}{n!} \sum_{k = 1} ^ {n} x_k B_{n, k} ( y).$$
    This implies that coordinates of the Bell product $ x \circ  y$ are equal to elements of a composition of two functions $x(y(t)),$ where
    $$x(t) = \sum_{n \geq 1} x_n \frac{t ^ n}{n!},$$
    $$y(t) = \sum_{n \geq 1} y_n \frac{t ^ n}{n!}.$$
    Specifically, it follows that the following relation holds
    \begin{equation}\label{iso_comp}
        ( x \circ  y)_n = \left[ \frac{t ^ n}{n!}\right] x(y(t)).
    \end{equation}
    Because function composition is associative, so is the Bell product. The identity element is $ e = (1, 0, 0, \ldots)$. Indeed,
    $$ x \circ  e =  x,$$
    $$ e \circ  x =  x.$$
    The inverse element $ x ^ {\langle -1 \rangle}$ should satisfy the following relation
    $$ x \circ  x^{\langle -1 \rangle} =  x^{\langle -1 \rangle} \circ  x =  e.$$
    Thus, the coordinates of the inverse element can be computed as
    $$\left( x \circ  x ^ {\langle -1 \rangle}\right)_1 = B_1( x,  x ^ {\langle -1 \rangle}) = 1.$$
    Then,
    $$\left( x ^ {\langle -1 \rangle}\right)_1 = \frac{1}{x_1}.$$
    Also,
    $$\left( x \circ  x ^ {\langle -1 \rangle}\right)_n = B_n\left( x,  x ^ {\langle -1 \rangle}\right) = 0.$$
    Further,
    $$\sum_{k = 1} ^ {n} x_k B_{n, k} \left( x ^ {\langle -1 \rangle}\right) = 0.$$
    Hence,
    $$x_1 x_n ^ {\langle -1 \rangle} + \sum_{k = 2} ^ {n} x_k B_{n, k} ( x ^ {\langle -1 \rangle}) = 0.$$
    Thus,
    \begin{equation} \label{lagrange_rec}
          x_n ^ {\langle -1 \rangle} = \frac{-1}{x_1} \sum_{k = 2} ^ {n} x_k B_{n, k} ( x ^ {\langle -1 \rangle}).
    \end{equation}

    In summary, the Bell product is associative, and there exist identity and inverse elements in the Bell group. Hence, the Bell group is a group.
\end{proof}

A formal exponential series $f(t)$ is said to be \emph{invertible} if there exists a formal exponential series $g(t)$ such that $f(g(t)) = g(f(t)) = t$. An obvious yet important remark is that from the equation (\ref{iso_comp}), it follows that Bell Groups are isomorphic to the groups of invertible formal exponential series with no constant term when equipped with functional composition.

An important tool we will use in proving the main result is the recursive Lagrange Inversion, which uses values \( x_1 ^ {\langle -1 \rangle}, x_2 ^ {\langle -1 \rangle}, \cdots, x_{n  - 1} ^ {\langle -1 \rangle} \) to compute \( x ^ {\langle -1 \rangle}_n.\)

\begin{lm}\label{lag_rec} (Recursive Lagrange Inversion) Let $x(t) = \sum_{n \geq 1} x_n\frac{t ^ n}{n!}$. Then, the coefficients of its compositional inverse $x^{\langle -1 \rangle}_n$ can be computed using the following recursive relation.
    \[
     x ^{\langle -1 \rangle}_n =
    \begin{cases}
      \dfrac{1}{x_1}, & \text{if } n = 1, \\\\
      \displaystyle \dfrac{-1}{x_1} \sum_{k = 2} ^ {n} x_k B_{n, k} (x ^ {\langle -1 \rangle}), & \text{if } n > 1.
    \end{cases}
\]
\end{lm}

\begin{proof}
    The proof is the same as the derivation of equation (\ref{lagrange_rec}) in the proof of Theorem \ref{bell_group}.
\end{proof}

Alternatively, it is possible to find the inverse of an element in the Bell Group using a closed-form expression (see p. 151 in \cite{ComtetBook}). 

\begin{lm}\label{comtet} Let $x(t) = \sum_{n \geq 1} x_n\frac{t ^ n}{n!}$. Then, the coefficients $x_n ^ {\langle -1 \rangle}$ of its compositional inverse $x^{\langle -1 \rangle}(t)$ can be obtained as

    \[
        x^{\langle -1 \rangle}_n =
        \begin{cases}
          \dfrac{1}{x_1}, & \text{if } n = 1, \\\\
          \displaystyle \sum_{k = 1} ^ {n - 1} \frac{(-1) ^ k}{x_1 ^ {n + k}} B_{n + k - 1, k}(0, x_2, x_3, \cdots), & \text{if } n > 1.
        \end{cases}
    \]
\end{lm}

We will utilize Lemma \ref{comtet} for the explicit computation of the coefficients of generating functions given as compositional inverses.

\section{Labeled Multipartite Series-reduced Trees}

In this section, we discuss the enumeration of rooted multipartite labeled series-reduced trees. Throughout this section, we consider trees to be \emph{non-planar}, meaning that the order of branches does not matter and swapping them results in the same tree.

A tree is said to be \emph{multipartite} if there are no two adjacent vertices that share the same color. By definition, rooted labeled multipartite series-reduced trees have uncolored labeled leaves and colored unlabeled inner vertices.

Additionally, out-degrees of inner vertices of rooted labeled multipartite series-reduced trees are not equal to $1.$ For convenience, sometimes we may call these trees \emph{labeled $m\text{-partite}$ trees}. 

We obtain and prove the generating function for labeled $m\text{-partite}$ series-reduced trees. First, let's define the weight of a labeled $m\text{-partite}$ series-reduced tree.

\begin{dfn} Let $T$ be a labeled $m\text{-partite}$ series-reduced tree with $s$ leaves, and let $V(T)$ denote its vertex set. Further, for all $v \in V(T)$, define $d(v)$ as the out-degree of $v$, and $col(v)$ as the color of the vertex $v$. The weight of a vertex $v$ is denoted as $x_{col(v), d(v)}$. Then, the weight $w(T)$ of a labeled $m\text{-partite}$ series-reduced tree $T$ is defined as a product of weights of its vertices.

\[
w(T) = \dfrac{t ^ s }{s!} \prod_{\substack{v \in V(T) \\ d(v)\neq 0}} x_{col(v), d(v)}.
\]
    
\end{dfn}

For example, consider the tree illustrated in Figure 2, and let red correspond to color $1$, blue correspond to color $2$, and green correspond to color $3$. The tree contains $10$ leaves, one red vertex with out-degree $3$, one red vertex with out-degree $2$, one blue vertex with out-degree $3$, and two green vertices with out-degree $3$. 
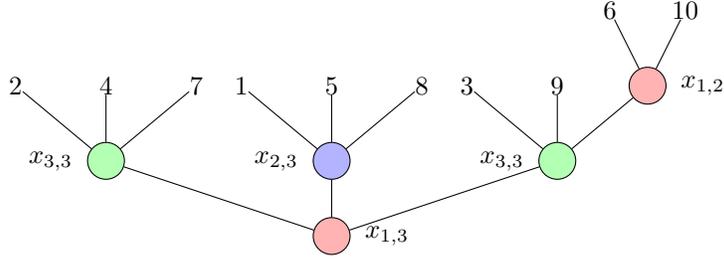
\begin{figure}[ht]
    \centering
    \begin{tikzpicture}[
        grow'=up,
        level distance=1cm,
        level 1/.style={sibling distance=3cm},
        level 2/.style={sibling distance=1.2cm},
        level 3/.style={sibling distance=1cm},
        inn/.style={circle,draw,scale=0.7,minimum size=7mm},
        rednode/.style={inn,fill=red!30},
        greennode/.style={inn,fill=green!30},
        bluenode/.style={inn,fill=blue!30},
        node_invisible/.style={inner sep=0pt,minimum size=0pt}
    ]
    \node[rednode, label=right:{$x_{1,3}$}] {}
        child { node[greennode, label=left:{$x_{3,3}$}] {}
            child { node[node_invisible] {2} }
            child { node[node_invisible] {4} }
            child { node[node_invisible] {7} }
        }
        child { node[bluenode, label=left:{$x_{2,3}$}] {}
            child { node[node_invisible] {1} }
            child { node[node_invisible] {5} }
            child { node[node_invisible] {8} }
        }
        child { node[greennode, label=left:{$x_{3,3}$}] {}
            child { node[node_invisible] {3} }
            child { node[node_invisible] {9} }
            child { node[rednode, label=right:{$x_{1,2}$}] {}
                child { node[node_invisible] {6} }
                child { node[node_invisible] {10} }
            }
        };
    \end{tikzpicture}
    \caption{A labeled $3$-partite series-reduced tree with inner vertex weights.}
\end{figure}

Thus, the weight of the tree is given as

\[
w(T) = \dfrac{t ^ {10}}{10!} x_{1, 3} x_{1, 2} x_{2, 3} x_{3, 3} ^ 2.
\]

\begin{dfn} Let \(P(m ,t, x) = P(m, t, x_{1, 1}, x_{1, 2}, \cdots x_{2, 1}, x_{2, 2}, \cdots x_{m, 1}, x_{m, 2}, \cdots)\) be the generating function for labeled $m\text{-partite}$ series-reduced trees. We define it in terms of the weights of trees as

    \[
    P(m, t,  x) = \sum_{T} w(T),
    \]
where the sum is over all labeled $m\text{-partite}$ series-reduced trees $T$.
\end{dfn}

\begin{proof}[\bf Proof of Theorem \ref{main_result}.]
    Let $P((m, c), t,  x)$ denote the generating function for labeled $m\text{-partite}$ trees with roots colored in color $c$. Similarly, let $P_s((m, c), x) = P_s((m, c), x_1, x_2, x_3, \cdots)$ denote the generating function for trees with $s$ leaves and roots colored in color $c$. In other words,

    \[
    P((m, c), t,  x) = \sum_{s \geq 1} \dfrac{t ^ s}{s!} P_s((m, c), x).
    \]

    If $P_ s(m, x_1, x_2, x_3, \cdots) = \left[ \dfrac{t ^s }{s!} \right] P(m, t,  x)$ is the generating function for labeled $m\text{-partite}$ trees with $s$ leaves and roots of any color, it's obvious that $\sum_{k = 1} ^ c P_s ((m, c), x) = P_ s (m,  x)$. 
    
    We can now express $P_s ((m, c), x)$ as a recurrence relationship. First, we enumerate forests of $m\text{-partite}$ trees with \( s \) leaves, where no tree has a root of color \( c \). Then, we will attach the roots of trees in the forest to a common root. The common root has color $c$, and its out-degree $k$ equals the number of trees in the forest. Thus, we obtain an $m\text{-partite}$ tree with $s$ leaves and root color $c$.

    First, the number of forests consisting of \( k \) trees with a total of \( s \) leaves, where none of the tree roots are colored with color \( c \), can be computed using the Bell polynomial.

    \[
    B_{s,k} \left( \left\{ P_j(m, x) - P_j((m, c), x) \right\}_{j=1}^{s - k + 1} \right).
    \]

    Then, attaching the trees in a forest to a common root colored \( c \) produces an $m\text{-partite}$ tree, from which

    \[
    P_s ((m, c), x) = \sum_{k = 2} ^ s x_{c, k} B_{s,k} \left( \left\{ P_j(m, x) - P_j((m, c), x) \right\}_{j=1}^{s - k + 1} \right).
    \]
    We can rewrite the left-hand side in terms of the Bell product.
    \[ \sum_{k = 2} ^ s x_{c, k}B_{s,k} \left( \left\{ P_j(m, x) - P_j((m, c), x) \right\}_{j=1}^{s - k + 1} \right) \]
    \[ = B_s \left(x_c, \left\{ P_j(m, x) - P_j((m, c), x) \right\}_{j=1}^{s} \right) - P_s(m, x) + P_s((m, c), x)
    \]
    \[
    = B_s(x_c, P(m, x) - P((m, c), x)) - P_s(m, x) + P_s((m, c), x).
    \]
    Hence, 
    \[
    P_s(m, x) = B_s(x_c,  P_m( x) - P_{m, c}( x)).
    \]
    Further,
    \[
    P(m, t,  x) = x_c(t) \circ (P(m, t,  x) - P((m, c),t,  x)).
    \]
    The composition of functions satisfies the left distributive law, so
    \[
    x_c ^ {\langle -1 \rangle}(t) \circ P(m, t,  x) = (P(m, t,  x) - P((m, c), t,  x)).
    \]
    We sum over all $1 \leq c \leq m$, and get the following expression
    \[
    \left( \sum_{c = 1} ^ m x_c ^ {\langle -1 \rangle}(t) \right) \circ P(m, t,  x) = (m - 1)P(m, t,  x).
    \]
    Using the Bell product, we can expand as
    \[
    P_s(m, x) m + \sum_{k = 2} ^ s B_{s, k} (P(m, x)) \sum_{c = 1} ^ {m} \left(  x_{c} ^ {\langle -1 \rangle} \right)_k = (m - 1) P_s(m, x).
    \]
    Then,
    \begin{equation}\label{wow}
    -\sum_{k = 2} ^ s B_{s, k} (P(m, x)) \sum_{c = 1} ^ {m} \left(  x_{c} ^ {\langle -1 \rangle} \right)_k = P_s(m, x).
    \end{equation}
    Notice that equation (\ref{wow})is the recursive Lagrange Inversion in Lemma \ref{lag_rec} when the coefficient of $t$ in the function equals $1$. Thus,
    \[
    P(m, t,  x) = \left(t + \sum_{c = 1} ^ m (x_{c} ^ {\langle -1 \rangle}(t) - t)\right) ^ {\langle -1 \rangle}
    \]
    which proves the theorem.

\end{proof}

\noindent\textbf{Example.} First 5 terms of $P(m, t, {x})$ for $m = 1, 2, 3$.
\[
\resizebox{\textwidth}{!}{$%
\setlength{\arraycolsep}{0pt}
\renewcommand{\arraystretch}{1.2}
\begin{array}{@{}r@{\;\;=\;\;}l@{}}
P(1, t, x) &
  \displaystyle
  t + \frac{t^{2}}{2!}\,x_{1,2}
  + \frac{t^{3}}{3!}\,x_{1,3}
  + \frac{t^{4}}{4!}\,x_{1,4}
  + \frac{t^{5}}{5!}\,x_{1,5} 
  + \cdots ,
\\[10pt]
P(2, t, x) &
  \begin{aligned}[t]
     & t + \frac{t^{2}}{2!}\bigl(x_{1,2}+x_{2,2}\bigr) \\[4pt]
     &+ \frac{t^{3}}{3!}\bigl(
          x_{1,3}+x_{2,3}+6x_{1,2}x_{2,2}
        \bigr) \\[4pt]
     &+ \frac{t^{4}}{4!}\Bigl(
          x_{1,4}+x_{2,4}
          +10\bigl(x_{1,3}x_{2,2}+x_{2,3}x_{1,2}\bigr)
          +15\bigl(x_{1,2}^{2}x_{2,2}+x_{1,2}x_{2,2}^{2}\bigr)
        \Bigr) \\[4pt]
     &+ \frac{t^{5}}{5!}\Bigl(
          x_{1,5}+x_{2,5}
          +20\,x_{1,3}x_{2,3} \\[2pt]
    &\quad
          +15\bigl(
              x_{1,4}x_{2,2}+x_{2,4}x_{1,2}
              +x_{2,2}x_{1,2}^{3}+x_{2,2}^{3}x_{1,2}
            \bigr) \\[2pt]
    &\quad
          +45\bigl(x_{2,3}x_{1,2}^{2}+x_{1,3}x_{2,2}^{2}\bigr)
          +60\bigl(x_{1,3}x_{2,2}x_{1,2}+x_{2,2}x_{2,3}x_{1,2}\bigr) \\[2pt]
    &\quad
          +180\,x_{2,2}^{2}x_{1,2}^{2}
        \Bigr) + \cdots ,
  \end{aligned}
\\[10pt]
P(3,t, x) &
  \begin{aligned}[t]
     & t + \frac{t^{2}}{2!}\bigl(x_{1,2}+x_{2,2}+x_{3,2}\bigr) \\[4pt]
     &+ \frac{t^{3}}{3!}\bigl(
          x_{1,3}+x_{2,3}+x_{3,3}
          + 6\!\bigl(x_{1,2}x_{2,2}+x_{1,2}x_{3,2}+x_{2,2}x_{3,2}\bigr)
        \bigr) \\[4pt]
     &+ \frac{t^{4}}{4!}\Bigl(
          15\bigl(
          x_{1,2}^2 x_{2,2}
      + x_{1,2} x_{2,2}^2
      + x_{1,2}^2 x_{3,2}
      + x_{1,2} x_{3,2}^2
      + x_{2,2}^2 x_{3,2}
      + x_{2,2} x_{3,2}^2
          \bigr) \\[2pt]
     &\quad
          + 10\bigl(
              x_{2,3}x_{1,2}+x_{3,3}x_{1,2}
            + x_{1,3}x_{2,2}+x_{1,3}x_{3,2}
            + x_{2,3}x_{3,2}+x_{2,2}x_{3,3}
          \bigr) \\[2pt]
     &\quad
          + 90\,x_{1,2}x_{2,2}x_{3,2}
          + x_{1,4}+x_{2,4}+x_{3,4}
        \Bigr) + \cdots
  \end{aligned}
\end{array}
$}
\]

The meaning of $\left[ \frac{t ^ s}{s!}\right] P(m, t,  x)$ is the sum of weights of all rooted $m\text{-partite}$ trees labeled series-reduced trees with $s$ leaves. For example,$\left[ \frac{t ^ 3}{3!}\right] P_2(t,  x) = x_{1, 3} + x_{2, 3} + 6 x_{1, 2} x_{2, 2}$ is the sum of weights the trees below.
\begin{figure}[htb]
  \centering
  \renewcommand{\arraystretch}{0.9} 
  \begin{adjustbox}{max width=\textwidth,max height=0.45\textheight,keepaspectratio}
    \begin{tabular}{@{}cccc@{}}
      \begin{tikzpicture}[scale=0.55, grow'=up, level 1/.style={sibling distance=1.3cm}]
        \node[node_black, scale=3.0] (x23) {}
          child {node[node_invisible]{1}}
          child {node[node_invisible]{2}}
          child {node[node_invisible]{3}};
        \node[scale=0.9] at ($(x23)+(2.4em, 0em)$) {$x_{2,3}$};
      \end{tikzpicture}
      &
      \begin{tikzpicture}[scale=0.55, grow'=up, level 1/.style={sibling distance=1.5cm}, level 2/.style={sibling distance=1.1cm}]
        \node[node_black, scale=3.0] (x22) {}
          child {node[node_invisible]{1}}
          child {node[node_white, scale=3.0] (x12){}
            child {node[node_invisible]{2}}
            child {node[node_invisible]{3}}};
        \node[scale=0.9] at ($(x22)+(2.4em, 0em)$) {$x_{2,2}$};
        \node[scale=0.9] at ($(x12)+(2.4em, 0em)$) {$x_{1,2}$};
      \end{tikzpicture}
      &
      \begin{tikzpicture}[scale=0.55, grow'=up, level 1/.style={sibling distance=1.5cm}, level 2/.style={sibling distance=1.1cm}]
        \node[node_black, scale=3.0] (x22){}
          child {node[node_invisible]{2}}
          child {node[node_white, scale=3.0] (x12){}
            child {node[node_invisible]{1}}
            child {node[node_invisible]{3}}};
        \node[scale=0.9] at ($(x22)+(2.4em, 0em)$) {$x_{2,2}$};
        \node[scale=0.9] at ($(x12)+(2.4em, 0em)$) {$x_{1,2}$};
      \end{tikzpicture}
      &
      \begin{tikzpicture}[scale=0.55, grow'=up, level 1/.style={sibling distance=1.5cm}, level 2/.style={sibling distance=1.1cm}]
        \node[node_black, scale=3.0] (x22){}
          child {node[node_invisible]{3}}
          child {node[node_white, scale=3.0](x12){}
            child {node[node_invisible]{1}}
            child {node[node_invisible]{2}}};
        \node[scale=0.9] at ($(x22)+(2.4em, 0em)$) {$x_{2,2}$};
        \node[scale=0.9] at ($(x12)+(2.4em, 0em)$) {$x_{1,2}$};
      \end{tikzpicture}
      \\[0.25cm]
      \begin{tikzpicture}[scale=0.55, grow'=up, level 1/.style={sibling distance=1.3cm}]
        \node[node_white, scale=3.0] (x13) {}
          child {node[node_invisible]{1}}
          child {node[node_invisible]{2}}
          child {node[node_invisible]{3}};
        \node[scale=0.9] at ($(x13)+(2.4em, 0em)$) {$x_{1,3}$};
      \end{tikzpicture}
      &
      \begin{tikzpicture}[scale=0.55, grow'=up, level 1/.style={sibling distance=1.5cm}, level 2/.style={sibling distance=1.1cm}]
        \node[node_white, scale=3.0] (x12) {}
          child {node[node_invisible]{1}}
          child {node[node_black, scale=3.0] (x22) {}
            child {node[node_invisible]{2}}
            child {node[node_invisible]{3}}};
        \node[scale=0.9] at ($(x22)+(2.4em, 0em)$) {$x_{2,2}$};
        \node[scale=0.9] at ($(x12)+(2.4em, 0em)$) {$x_{1,2}$};
      \end{tikzpicture}
      &
      \begin{tikzpicture}[scale=0.55, grow'=up, level 1/.style={sibling distance=1.5cm}, level 2/.style={sibling distance=1.1cm}]
        \node[node_white, scale=3.0] (x12) {}
          child {node[node_invisible]{2}}
          child {node[node_black, scale=3.0] (x22) {}
            child {node[node_invisible]{1}}
            child {node[node_invisible]{3}}};
        \node[scale=0.9] at ($(x22)+(2.4em, 0em)$) {$x_{2,2}$};
        \node[scale=0.9] at ($(x12)+(2.4em, 0em)$) {$x_{1,2}$};
      \end{tikzpicture}
      &
      \begin{tikzpicture}[scale=0.55, grow'=up, level 1/.style={sibling distance=1.5cm}, level 2/.style={sibling distance=1.1cm}]
        \node[node_white, scale=3.0] (x12) {}
          child {node[node_invisible]{3}}
          child {node[node_black, scale=3.0] (x22) {}
            child {node[node_invisible]{1}}
            child {node[node_invisible]{2}}};
        \node[scale=0.9] at ($(x22)+(2.4em, 0em)$) {$x_{2,2}$};
        \node[scale=0.9] at ($(x12)+(2.4em, 0em)$) {$x_{1,2}$};
      \end{tikzpicture}
    \end{tabular}
  \end{adjustbox}
  \caption{Rooted bipartite labeled series-reduced trees with $3$ leaves.}
  \label{fig:series_reduced_3leaves}
\end{figure}
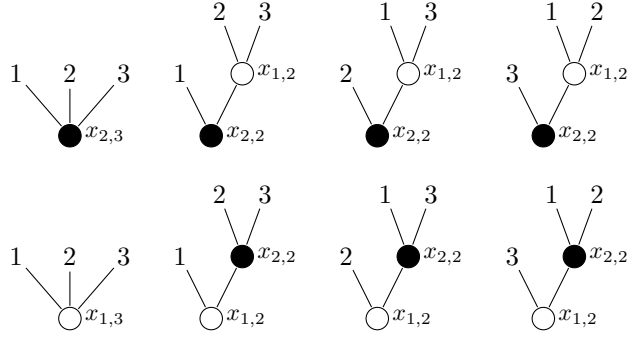

\newpage

The generating function for labeled $m$-partite series-reduced trees $P_s(m, x)$ is precisely the coefficient $\left[ \frac{t ^ s}{s!}\right] P(m, t,  x)$. We can use the Lemma \ref{comtet} to obtain a closed-form expression for the coefficients of $P(m, t,  x)$.

\begin{thm}\label{closed_form} The closed form expression of the generating function for labeled $m\text{-partite}$ series-reduced trees with $s$ leaves $P_s(m, x)$ is can be given by

    \[
    P_s(m, x) = \sum_{k = 0} ^ {s} (-1) ^ {k} B_{s + k - 1, k} \left(0, \sum_{c = 1} ^ m x_{c, 2} ^{\langle -1 \rangle}, \sum_{c = 1} ^ m x_{c, 3} ^{\langle -1 \rangle}, \cdots, \sum_{c = 1} ^ m x_{c, s} ^{\langle -1 \rangle} \right),
    \]
where $ x_{c, j} ^ {\langle -1 \rangle}$ is the coefficient of the series expansion of the compositional inverse of the degree function $x_c^{\langle -1 \rangle}(t)$, which can be expressed explicitly by

\[
 x_{c, j} ^ {\langle -1 \rangle} =  \sum_{k = 0} ^ {j} (-1) ^ {k} B_{j + k - 1, k}(0, x_{c, 2}, x_{c, 3}, \ldots, x_{c, j}).
\]
    
\end{thm}

\begin{proof}
    By Theorem \ref{main_result}, we know that

    \[
    P(m, t,  x) = \left(t + \sum_{c = 1} ^ {m} \left(x_c ^ {\langle -1 \rangle} (t) - t\right)  \right) ^ {\langle -1 \rangle}.
    \]

    To compute coefficients of \( \dfrac{t ^ s}{s!} \) of $P(m, t,  x)$, we use Lemma \ref{comtet} (closed-form Lagrange Inversion formula). It is evident that

    \[
    \left[t \right] P(m, t,  x) = \dfrac{1}{(1 - m) + \sum_{c = 1} ^ m \left(  x_{c} ^ {\langle -1 \rangle} \right)_1} = \dfrac{1}{(1 - m) + m} = 1 .
    \]

    Therefore, we can assume that  $s > 1$. Then,

    \[
    P_s(m, x) = \left[ \dfrac{t ^s }{s!}\right] P(m, t,  x) = \sum_{k = 1} ^ {s} (-1) ^ {k} B_{s + k - 1, k} \left( 0, \sum_{c = 1} ^ m x_{c, 2} ^{\langle -1 \rangle}, \sum_{c = 1} ^ m x_{c, 3} ^{\langle -1 \rangle}, \cdots \right).
    \]

    By definition of Bell Polynomials, $B_{0, 0}(x) = 1$ and $B_{n, 0}(x) = 0$ for $n > 0$. This is why, we can let $k$ in the sum start from $k = 0$ to account for the case when $s = 1.$

    \[
    P_s(m, x) = \sum_{k = 0} ^ {s} (-1) ^ {k} B_{s + k - 1, k} \left(0, \sum_{c = 1} ^ m x_{c, 2} ^{\langle -1 \rangle}, \sum_{c = 1} ^ m x_{c, 3} ^{\langle -1 \rangle}, \cdots, \sum_{c = 1} ^ m x_{c, s} ^{\langle -1 \rangle} \right).
    \]

    Similarly, it is possible to apply Lemma \ref{comtet} to find an explicit form of $  x_{c, i} ^ {\langle -1 \rangle}$. In particular,

    \begin{dmath*}
        x_{c, j} ^ {\langle -1 \rangle} = \left[ \dfrac{t ^ j}{j!}\right] x_c ^ {\langle -1 \rangle}(t) = \left[ \dfrac{t ^ j}{j!}\right] \left( t + \sum_{n \geq 2} x_{c, n} \dfrac{t ^ n}{n!}\right) ^ {\langle -1 \rangle} = \sum_{k = 0} ^ {j} (-1) ^ {j + k - 1} B_{j + k - 1, k}(0, x_{c, 2}, x_{c, 3}, \ldots, x_{c, j}).
    \end{dmath*}
\end{proof}

\section{Applications}

In this section, we will use Theorem \ref{main_result} and Theorem \ref{closed_form} to obtain several applications and corollaries such as the closed-form enumeration of labeled multipartite trees, labeled multipartite circular trees, and symbolic ultrametrics. We also obtain the computation of the number of increasingly labeled processes.

\subsection{Enumeration of labeled multipartite trees and Symbolic Ultrametrics}

An \textit{ultrametric} defined on a set $X$ is a map $d : X \times X \rightarrow \mathbb{R}$ satisfies the following conditions for all $x, y, z \in X.$

\begin{enumerate}
    \item $d(x, y) = 0$ if and only if $x = y,$
    \item $d(x, y) = d(y, x),$
    \item $d(x, z) \leq \max(d(x, y), d(y, z)).$
\end{enumerate}

A \textit{symbolic ultrametric}, introduced by Böcker and Dress \cite{Dress}, is defined as follows. For sets $X$ and $M$, a symbolic ultrametric is a function $D : X \times X \rightarrow M$ such that for all $x, y, z \in X$ the following conditions hold.

\begin{enumerate}
    \item $D(x, y) = D(y, x),$
    \item $\# \{D(x, y), D(x, z), D(y, z)\} \leq 2,$
    \item $\not\exists a, b, c, d \in X$ such that $D(a, b) = D(b, c) = D(c,d) \neq D(b, d) = D(d, a) = D(a, c).$
\end{enumerate}

In \cite{Dress}, Theorem 2 states that there exists a bijection between $m\text{-partite}$ labeled series-reduced trees on $s$ leaves and symbolic ultrametrics $D : X \times X \rightarrow M$ such that $|X| = s$ and $|M| = m$. A natural question is to enumerate such objects.

We find the explicit enumeration of $m$-partite labeled series-reduced trees and symbolic ultrametrics.

\begin{dfn}
    A derangement of length $n$ is a permutation $\sigma$ that has no fixed points. Formally, $\sigma$ is a derangement if there exists no $1 \leq i \leq n$ such that $\sigma_i = i$. A derangement $\sigma$ is said to have $k$ cycles if its cycle decomposition has $k$ cycles.
\end{dfn}

Below, we count the number of derangements of length $n$ and $k$ cycles.

\begin{lm}\label{derangements}
    Let $!n_k$ denote the number of derangements of length $n$ and $k$ cycles. Then, 

    \[
    !n_k = B_{n, k} (0, \{ (i - 1)!\}_{i = 2} ^ \infty).
    \]
\end{lm}

\begin{proof}
    By definition, Bell Polynomial $B_{n, k}( x)$ is the generating function for labeled partitions of a segment of length $n$ to $k$ parts. It is evident that $(i - 1)!$ is the number of distinct cycles of length $i$. Thus, substituting $x_i = (i - 1)!$ into $B_{n, k}( x)$ gives the number of permutations with $k$ cycles. However, by the definition of derangements, there are no fixed points (cycles of length $1$) allowed in the cycle decomposition. So, if we set $x_1 = 0$ and $x_i = (i - 1)!$, the number of derangements with $k$ cycles is $B_{n, k}( x)$.
\end{proof}

\begin{rmk}
    The number $!n_k$ is also known as the associated Stirling numbers of the first kind (see p. 75 of \cite{RiordanBook}).
\end{rmk}

\begin{proof}[\bf Proof of Corollary \ref{main_crl}]
    Set $x_{c, k} = 1$ for all $1 \leq c \leq m$ and $k \geq 1$. Then, the degree functions and their inverses are

    \[
    x_c(t) = \sum_{n \geq 1} \dfrac{t ^ n}{n!} = e^{t} - 1,
    \]
    \[
    x_c^{\langle -1 \rangle}(t) = \left( e^{t} - 1\right) ^ {\langle -1 \rangle} = \log(1 + t).
    \]
    Thus, by Theorem \ref{main_result}, we obtain
    \[
    A(m, t) = (t(1 - m) + m \log(1 + t)) ^ {\langle -1 \rangle}.
    \]
    Expand $x_c^{\langle -1 \rangle}(t) = \log(1 + t)$.
    \[
    x_c^{\langle -1 \rangle}(t) = \log(1 + t) = \sum_{n \geq 1} (-1) ^ {n - 1} (n - 1)! \dfrac{t ^ n}{n!}.
    \]
    Thus, we obtain that $x_{c, i} ^ {\langle -1 \rangle} = (-1) ^ {i - 1}(i - 1)!$. We can substitute these values into Theorem \ref{closed_form} and obtain the following relation
    \begin{dmath*}
        a_s(m) = \sum_{k = 0} ^ s (-1) ^ {k} B_{s + k - 1, k} (0, -m, 2m, \cdots, (-1) ^ {s - 1} (s - 1)! m) = (-1) ^ {s - 1} \sum_{k = 0} ^ s (-m) ^ k  B_{s + k - 1, k} (0, -m, 2m, \cdots, (-1) ^ {s - 1} (s - 1)! m).
    \end{dmath*}
    Applying Lemma \ref{derangements}, we get the explicit formula for the number of labeled $m$-partite series-reduced trees.
    \[
    a_s(m) = (-1) ^ {s - 1} \sum_{k = 0} ^ s (-m) ^ k  (!(s + k - 1)_k).
    \]
\end{proof}

\noindent When $m = 2,$ our sequence is the number of series-parallel networks with labeled edges \cite{RiordanSchroeder}.

\begin{crl} The number of bipartite labeled series-reduced trees can be given by the following relation

\[
a_s(2) = (-1) ^ {s - 1} \sum_{k = 0} ^ s (-2) ^ k (!(s + k - 1)_k).
\]

\end{crl}

If $\mathcal A(m, t) = \sum_{s \geq 0} a_s(m) \frac{t ^ s}{s!} = 1 + A(m, t)$ is the generating function for the number of $m\text{-partite}$ labeled series-reduced trees including an empty tree (a tree with no vertices), then we show that it satisfies the following relation.

\begin{prp} The generating function $\mathcal A(m, t)$ for the number of $m\text{-partite}$ labeled series-reduced trees with an empty tree satisfies the following relation

\[
\mathcal A(m, t) = 1 + (\mathcal A(m, t)) ^ m \int_{0} ^ {t} \left(\dfrac{1}{\mathcal A(m, t)}\right) ^ m dt.
\]
    
\end{prp}

\begin{proof}
    First, we will re-express the given expression
    \[
        \mathcal A(m, t) = 1 + (\mathcal  A(m, t)) ^ m \int_{0} ^ {t} \left(\dfrac{1}{\mathcal A(m, t)}\right) ^ m dt.
    \]
    Rearranging terms, we obtain
    \[
        \int_{0} ^ {t} \left(\dfrac{1}{\mathcal  A(m, t)}\right) ^ m dt = \dfrac{\mathcal  A(m, t) - 1}{(\mathcal  A(m, t)) ^ m}.
    \]
    Because of $\mathcal A(m, t) = 1 + A(m, t)$, we obtain
    \[
        \int_{0} ^ {t} \left(\dfrac{1}{1 + A(m, t)}\right) ^ m dt = \dfrac{A(m, t)}{(1 + A(m, t)) ^ m}.
    \]
    Differentiating both sides by $t$ leads to the following relation 
    \[
        \left(\dfrac{1}{1 + A(m, t)}\right) ^ m = \left( \dfrac{1 + (1 - m) A(m, t)}{(1 + A(m, t)) ^ {m + 1}} \right) \frac{\partial}{\partial t}A(m, t). 
    \]
    Multiplying $(1 + A(m, t)) ^ {m + 1}$ by both sides, we obtain a differential equation
    \[
    (1 + A(m, t)) = (1 + (1 - m) A(m, t)) \frac{\partial}{\partial t} A(m, t).
    \]
    We can separate the differential equation as follows
    \[
    \dfrac{\partial A(m, t)}{\partial t} = \dfrac{(1 + A(m, t))}{1 + (1 - m) A(m, t)}. 
    \]
    Thus,
    \[
    \int_0 ^ t \dfrac{1 + (1 - m)A(m, t)}{1 + A(m, t)} dA(m, t) = \int_0^t dt = t.
    \]
    Note that $A(m, 0) = 0$. Solving the integral on the left-hand side, we obtain
    \[
    (1 - m) A(m, t) + m \log(1 + A(m, t)) = t.
    \]
    One can rewrite this expression using a compositional inverse on $t$ as follows
    \[
    ((1 - m) (A(m, t)) + m \log(1 + A(m, t))) \circ A^ {\langle -1 \rangle}(m, t) = A ^ {\langle -1 \rangle}(m, t).
    \]
    From which we obtain the following expression
    \[
    A(m, t) = (t (1 - m) + m \log( 1 +t)) ^ {\langle -1 \rangle}.
    \]
    Notice that this equation is the generating function for the number of $m\text{-partite}$ labeled series-reduced trees in Corollary \ref{main_crl}. Thus, the given relation holds.
\end{proof}

A table for the number of symbolic ultrametrics is given below. In particular, row $m = 2$ gives a sequence for the number of series-parallel networks with labeled edges \cite{RiordanSchroeder}. The row corresponding to $m = 3$ yields a sequence that counts the number of increasingly labeled processes which are described in \cite{Bodini}. We will provide a proof of this result in Section \ref{parallel_computing}.

\begin{table}[htb]
\centering
\begin{tabular}{c|rrrrrrrr}
$m \backslash s$ & 1 & 2 & 3 & 4 & 5 & 6 & 7 & 8 \\
\hline
1 & 1 & 1 & 1 & 1 & 1 & 1 & 1 & 1 \\
2 & 1 & 2 & 8 & 52 & 472 & 5504 & 78416 & 1320064 \\
3 & 1 & 3 & 21 & 243 & 3933 & 81819 & 2080053 & 62490339 \\
4 & 1 & 4 & 40 & 664 & 15424 & 460576 & 16808320 & 724904896 \\
5 & 1 & 5 & 65 & 1405 & 42505 & 1653125 & 78578225 & 4414067725 \\
6 & 1 & 6 & 96 & 2556 & 95256 & 4563936 & 267253776 & 18494891136 \\
7 & 1 & 7 & 133 & 4207 & 186277 & 10603999 & 737769781 & 60662126959 \\
8 & 1 & 8 & 176 & 6448 & 330688 & 21804224 & 1757138048 & 167347010944 \\
\end{tabular}
\caption{Number of symbolic ultrametrics $D:X\times X \rightarrow M$ s.t. $|X|=s$, $|M|=m.$}
\label{tab:symb_values}
\end{table}

It is also possible to find the number of $m$-partite labeled series-reduced trees such that all vertices are colored. We call such trees \emph{fully colored $m\text{-partite}$ labeled series-reduced trees}, and denote the number of trees with $s$ leaves as $a_s^{\text{fc}}(m)$.

\begin{thm} Number of fully colored $m$-partite labeled series-reduced trees with $s > 1$ leaves, denoted as $\bar a_m ^ s$, satisfy the following relation
    \[
    a_s^{\text{fc}}(m) = (m - 1) ^ s a_s(m) =  -(1 - m) ^ {s} \sum_{k = 0} ^ s (-m) ^ k  (!(s + k - 1)_k).
    \]
\end{thm}

\begin{proof}
    By definition, leaves are vertices with out-degree $0$, and when $s > 1$, they have exactly one parent. So, there is $(m - 1)$ number of ways to color a leaf in an $m$-partite labeled series-reduced tree. Thus, $a_s ^ {\text{fc}}(m) = (m - 1) ^ s a_s(m)$. 
\end{proof}

A table for the number of fully colored $m\text{-partite}$ labeled series-reduced trees with \(s\) leaves is given below.

\begin{table}[h!]
\centering
\begin{tabular}{c|rrrrrr}
$m \backslash s$ & 1 & 2 & 3 & 4 & 5 & 6 \\
\hline
1 & 1 & 0 & 0 & 0 & 0 & 0 \\
2 & 2 & 2 & 8 & 52 & 472 & 5504 \\
3 & 3 & 12 & 168 & 3888 & 125856 & 5236416 \\
4 & 4 & 36 & 1080 & 53784 & 3748032 & 335759904 \\
5 & 5 & 80 & 4160 & 359680 & 43525120 & 6771200000 \\
6 & 6 & 150 & 12000 & 1597500 & 297675000 & 71311500000 \\
\end{tabular}
\caption{Number of fully colored $m\text{-partite}$ labeled series-reduced trees with \(s\) leaves.}
\end{table}

Another application of Theorem \ref{closed_form} concerns the number of $m\text{-partite}$ labeled series-reduced mobiles. A \emph{mobile}, or a circular tree, is a tree where branches are arranged around a vertex such that the order of branches is the same up to cyclic rotations.

\begin{prp}
    The generating function for the number of $m$-partite labeled series-reduced mobiles (circular rooted trees) can be expressed using the following relation

    \[
    G(m, t) = \left(t(1 - m) - m e ^ {-t} + m \right) ^ {\langle -1 \rangle}.
    \]

    An explicit formula for counting the number of $m$-partite labeled series-reduced mobiles can be expressed using associated Stirling numbers of the second kind $b(n, k)$.

    \[
    g_s(m) = (-1) ^ {s - 1} \sum_{k = 0} ^ s (-m) ^ k b(s + k - 1, k).
    \]

\end{prp}

\begin{proof}
    Set $x_{c, i} = (i - 1)!$ for all $1 \leq c \leq m$ and $i \geq 1$. Thus, the degree functions and their inverses are the following relations
    \[
    x_c(t) = \sum_{n \geq 1} \dfrac{t ^ n}{n} = -\log(1 - t),
    \]
    \[
    x_c^{\langle -1 \rangle}(t) = -e^{-t} + 1,
    \]
    \[
    \left( x_{c} ^ {\langle -1 \rangle}\right)_i = (-1) ^ {i + 1}.
    \]
    By Theorem \ref{main_result}, we obtain
    \[
    G(m, t) = \left(t(1 - m) + m (-e ^ {-t} + 1) \right) ^ {\langle -1 \rangle}.
    \]
    Additionally, we may plug the values of $x_{c, i} ^ {\langle -1 \rangle}$ into Theorem \ref{closed_form} to obtain the following  formula
    \begin{dmath*}
        g_s(m) = \sum_{k = 0} ^ s (-1) ^ k B_{s + k - 1, k} (0, -m, m, -m, m, \cdots, (-1) ^ {s + 1} m) = (-1) ^ {s - 1} \sum_{k = 0} ^ s (-m) ^ k B_{s + k - 1, k} (0, 1, 1, 1, \cdots).
    \end{dmath*}

    The associated Stirling numbers of the second kind (see p.76
    \cite{RiordanBook})
     are expressed as Bell Polynomials as follows

    \[
    b(n, k) = B_{n ,k} (0,1 , 1, 1, \cdots).
    \]
Using this definition, we obtain the following relation

    \[
    g_s(m) = (-1) ^ {s - 1} \sum_{k = 0} ^ s (-m) ^ k b(s + k - 1, k).
    \]
\end{proof}

Below we provide a table for the number of $m$-partite labeled mobiles (circular rooted trees) with $s$ leaves.

\begin{table}[h!]
\centering
\begin{tabular}{c|rrrrrrrr}
$m \backslash s$ & 1 & 2 & 3 & 4 & 5 & 6 & 7 & 8 \\
\hline
1 & 1 & 1 & 2 & 6 & 24 & 120 & 720 & 5040 \\
2 & 1 & 2 & 10 & 82 & 938 & 13778 & 247210 & 5240338 \\
3 & 1 & 3 & 24 & 318 & 5892 & 140304 & 4082712 & 140389824 \\
4 & 1 & 4 & 44 & 804 & 20556 & 675588 & 27135468 & 1288020708 \\
5 & 1 & 5 & 70 & 1630 & 53120 & 2225480 & 113950720 & 6895234480 \\
6 & 1 & 6 & 102 & 2886 & 114294 & 5819190 & 362107110 & 26628964710 \\
7 & 1 & 7 & 140 & 4662 & 217308 & 13022688 & 953817480 & 82561002048 \\
8 & 1 & 8 & 184 & 7048 & 377912 & 26052104 & 2195014072 & 218563826824 \\
\end{tabular}
\caption{Number of labeled multipartite mobiles (circular rooted trees).}
\end{table}
\newpage
It seems that $m = 2$ can also be expressed in terms of second-order Eulerian numbers OEIS \href{https://oeis.org/A112487}{A112487}.

\subsection{Multicolored chain-increasing binary trees}

In this section, we will show that labeled $(m + 1)\text{-partite}$ series-reduced trees are in bijection with the so-called $m\text{-colored}$ chain-increasing binary trees. This will be important in proving that the number of labeled $3\text{-partite}$ series-reduced trees is the same as the number of increasingly labeled processes.

We call a tree \emph{binary} if its vertices have at most $2$ children. A vertex is called a \emph{junction} if it has $2$ children. \emph{Chains} are the vertices with at most $1$ child. We call a tree $m\text{-colored}$ chain-increasing binary tree, if it satisfies the following three conditions

\begin{enumerate}
    \item Junctions are colored, but unlabeled,
    \item Chains are uncolored, but labeled,
    \item Labels are positioned along any branch from the root in increasing order.
\end{enumerate}

\noindent {\bf Example.} All eight $\text{1-colored}$ chain-increasing trees with $3$ chains. It is the same as the number $a_3(2) = 8$ of labeled bipartite series-reduced trees with $3$ leaves.

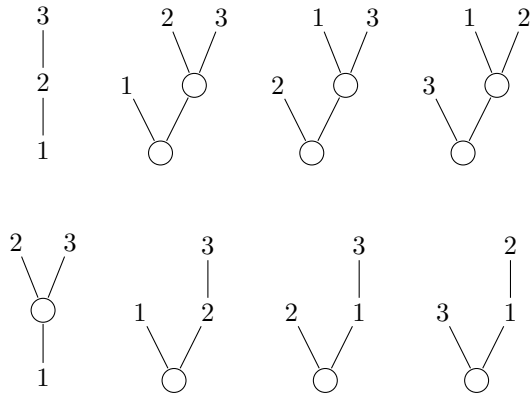
\begin{figure}[h!]
\centering
\begin{center}
    \begin{tabular}{ccccc}
    \begin{tikzpicture}[scale=0.6, grow'=up, level distance=1.5cm, every node/.style={scale=1}]
      \node[node_invisible]{1}
        child {node[node_invisible]{2}
          child {node[node_invisible]{3}}};
    \end{tikzpicture}
    &
    \begin{tikzpicture}[scale=0.6, grow'=up, level 1/.style={sibling distance=1.5cm}, level 2/.style={sibling distance=1.2cm}, every node/.style={scale=1}]
      \node[node_white, scale=3.2] {}
        child {node[node_invisible]{1}}
        child {node[node_white, scale=3.2] {}
          child {node[node_invisible]{2}}
          child {node[node_invisible]{3}}};
    \end{tikzpicture}
    &
    \begin{tikzpicture}[scale=0.6, grow'=up, level 1/.style={sibling distance=1.5cm}, level 2/.style={sibling distance=1.2cm}, every node/.style={scale=1}]
      \node[node_white, scale=3.2] {}
        child {node[node_invisible]{2}}
        child {node[node_white, scale=3.2] {}
          child {node[node_invisible]{1}}
          child {node[node_invisible]{3}}};
    \end{tikzpicture}
    &
    \begin{tikzpicture}[scale=0.6, grow'=up, level 1/.style={sibling distance=1.5cm}, level 2/.style={sibling distance=1.2cm}, every node/.style={scale=1}]
      \node[node_white, scale=3.2] {}
        child {node[node_invisible]{3}}
        child {node[node_white, scale=3.2] {}
          child {node[node_invisible]{1}}
          child {node[node_invisible]{2}}};
    \end{tikzpicture}
    
    \\[0.7cm]
    \begin{tikzpicture}[scale=0.6, grow'=up, 
        level 1/.style={sibling distance=1.5cm}, 
        level 2/.style={sibling distance=1.2cm}, 
        every node/.style={scale=1}]
      \node[node_invisible]{1}
        child {node[node_white, scale=3.2] (x1) {}
          child {node[node_invisible]{2}}
          child {node[node_invisible]{3}}};
    \end{tikzpicture}
    &
    \begin{tikzpicture}[scale=0.6, grow'=up, 
        level 1/.style={sibling distance=1.5cm}, 
        level 2/.style={sibling distance=1.2cm}, 
        every node/.style={scale=1}]
      \node[node_white, scale=3.2] (x1) {}
        child {node[node_invisible]{1}}
        child {node[node_invisible]{2}
          child {node[node_invisible]{3}}};
    \end{tikzpicture}
    &
    \begin{tikzpicture}[scale=0.6, grow'=up, 
        level 1/.style={sibling distance=1.5cm}, 
        level 2/.style={sibling distance=1.2cm}, 
        every node/.style={scale=1}]
      \node[node_white, scale=3.2] (x1) {}
        child {node[node_invisible]{2}}
        child {node[node_invisible]{1}
          child {node[node_invisible]{3}}};
    \end{tikzpicture}
    &
    \begin{tikzpicture}[scale=0.6, grow'=up, 
        level 1/.style={sibling distance=1.5cm}, 
        level 2/.style={sibling distance=1.2cm}, 
        every node/.style={scale=1}]
      \node[node_white, scale=3.2] (x1) {}
        child {node[node_invisible]{3}}
        child {node[node_invisible]{1}
          child {node[node_invisible]{2}}};
    \end{tikzpicture}
    \end{tabular}
\end{center}
\caption{Chain-increasing binary trees with $3$ chains.}
\end{figure}

If we color junctions in $m$ colors, one can see that the number of $m\text{-colored}$ chain-increasing binary trees with $3$ chains is $m ^ 0 + m ^ 2 + m ^ 2 + m ^ 2 + m + m + m + m = 1 + 4m + 3m ^ 2$.

\,

Let $y_s(m)$ denote the number of $m\text{-colored}$ chain-increasing binary trees with $s$ chains. Then, one can count the number of these trees using the following recurrence relation.

\begin{lm}\label{bin_enum}
    Let $y_1(m) = 1$. The number $y_s(m)$ of $m\text{-colored}$ chain-increasing binary trees with $s > 1$ chains can be computed using the following recurrence relation
    \[
    y_s(m) = y_{s - 1}(m) + m B_{s, 2} (y(m)).
    \]
\end{lm}

\begin{proof}
    It is obvious $y_1(m) = 1$. Then, for $s > 1$, the root of a tree may either be a unary or a binary vertex. If the root is unary, then, the number of trees is $y_{s - 1}(m)$. If the root is binary, there is $m$ ways to color it and $B_{s, 2}(y(m))$ ways to attach trees with $s$ chains in total. Thus, we obtain that 

    \[
    y_s(m) = y_{s - 1}(m) + m B_{s, 2} (y(m)).
    \]
\end{proof}

Now, we show that $y_s(m - 1)$ is the same as $a_s(m)$ the number of labeled $m\text{-partite}$ series-reduced trees.

\begin{thm}\label{binary=series-reduced}
    Let $y_s(m - 1)$ be the number of $(m - 1)\text{-colored}$ chain-increasing binary trees and $a_s(m)$ the number of labeled $m\text{-partite}$ series-reduced trees. Then,

    \[
    y_s(m - 1) = a_s(m ).
    \]
\end{thm}

\begin{proof}
    By Lemma \ref{bin_enum}, the sequence $y_s(m -1)$ satisfies the following relation 
\begin{equation}\label{recc_gamma}
        y_{s}(m - 1) = y_{s - 1}(m - 1) + (m - 1) B_{n, 2}(y(m - 1)).
    \end{equation}
    Let $Y(m, t)$ be its exponential generating function

    \[
    Y(m , t) = \sum_{s \geq 1} y_s(m) \frac{t^s}{s!}.
    \]
 Then (\ref{recc_gamma}) can be written in terms of generating functions 

    \begin{equation}\label{gamma_expgenfunc}
        Y(m - 1, t) = t + \sum_{s \geq 2} y_{s - 1}( m - 1) \frac{t ^ s}{s!} + (m - 1) \sum_{s \geq 2} B_{s, 2}(y(m - 1)) \frac{t ^ s}{s!}.
    \end{equation}
    Recall that the Bell product of two functions $f(t) = \sum_{n \geq 1} f_n \frac{t ^ n}{n!}$ and $g(t) = \sum_{n \geq 1} g_n \frac{t ^ n}{n!}$ can be given by 
    \[
    (f \circ g)(t) = \sum_{n \geq1 } \frac{t ^ n}{n!} \sum_{k = 1} ^ n f_k B_{n, k}(g_1, g_2, \cdots, g_{n - k + 1}).
    \]
    Therefore,
    \[
    \frac{Y(m - 1, t) ^ 2}{2} = \frac{t ^ 2}{2} \circ Y(m - 1, t) = \sum_{s \geq 2} B_{s, 2}(y(m - 1)) \frac{t ^ s}{s!}.
    \]
    Further,
    \[
    \sum_{s \geq 2} y_{s - 1}(m - 1) \frac{t ^ s}{s!} = \int_{0} ^ t Y(m - 1 ,t) dt.
    \]
    Substituting into (\ref{gamma_expgenfunc}) gives us 
    \[
    Y(m - 1, t) = t + \int_{0} ^ t Y(m - 1 ,t) dt + (m - 1)\frac{Y(m - 1, t) ^ 2}{2}.
    \]
    and
    \[
    \frac{\partial}{\partial t} Y(m - 1, t) = 1 + Y(m - 1, t) + (m - 1) Y(m - 1, t) \frac{\partial}{\partial t} Y(m - 1, t).
    \]
    Thus,
    \[
    \frac{\partial}{\partial t} Y(m - 1, t) = \frac{1 + Y(m - 1, t)}{1 + (1 - m) Y(m - 1, t)}.
    \]
    We have
    \[
    t = \int_0 ^ t \frac{1 + (1 - m ) Y(m - 1, t)}{1 + Y(m - 1, t)} dY(m - 1, t) = 
    \]
    \[
    \int_0 ^ t \left( 1 - \frac{m Y(m - 1, t)}{1 + Y(m - 1, t)}\right) dY(m - 1, t) = 
    \]
    \[
    Y(m - 1, t) - m \int_0 ^ t \frac{Y(m - 1, t)}{1 + Y(m - 1, t)} dY(m - 1, t) = 
    \]
    \[
    (1 - m) Y(m - 1, t) + m \log(1 + Y(m -1, t)).
    \]
    So we have established that
    \begin{equation}\label{t=relinY}
        t = (1 - m) Y(m - 1, t) + m \log(Y(m - 1, t)).
    \end{equation}
    Now, let us calculate the Bell product $t \circ f ^ {\langle -1 \rangle}(t)$. Since the composition of functions satisfies the left distributive law
    \[
    (f + g) \circ h = f \circ h + g \circ h,
    \]
    by (\ref{t=relinY})  we obtain
    \[
    Y ^ {\langle -1 \rangle}(m - 1, t) = t \circ Y ^ {\langle -1 \rangle}(m - 1, t) = 
    \]
    \[
    (1 - m) Y(m - 1, t) + m \log(Y(m - 1, t)) \circ Y ^ {\langle -1 \rangle} (m - 1, t) = 
    \]
    \[
    (1 - m) t + m \log(1 + t).
    \]
    By Corollary \ref{main_crl}, we have 
    \[
    Y(m - 1, t) = A(m, t).
    \]
  In other words, 
    \[
    y_s(m - 1) = a_s(m).
    \]
\end{proof}

\subsection{Increasing Parallel Processes}\label{parallel_computing}

Pure parallel processes are defined using the following grammar

\begin{enumerate}
    \item An atomic action denoted by $a, b, c, \cdots$ is a process,
    \item Prefixing $a.P$ of an action $a$ and process $P$ is a process,
    \item The composition $P_1 \parallel P_2$ is a process.
\end{enumerate}

Processes built from these rules can be expressed in terms of a \textit{process tree}. The size of a process tree is the number of actions in a tree.

Assigning increasing labels on the process tree is called an increasingly labeled process. It was shown in \cite{Bodini} that there is an isomorphism between the number of computations and increasingly labeled processes.

Below is an \textit{increasingly labeled} process tree of size $5$. The process is $(a.b) \parallel (c \parallel (e.d))$.

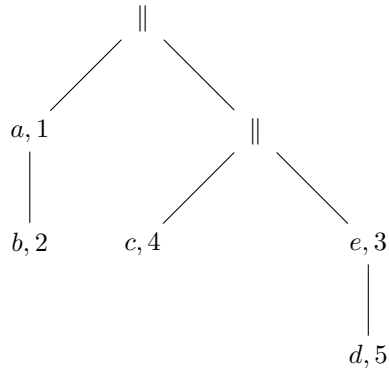
\begin{figure}[ht]
\begin{center}
    \begin{tikzpicture}[sibling distance=30mm, level distance=15mm, 
  every node/.style={font=\itshape}]
  \node {$\parallel$}
    child {node {$a, 1$}
      child {node {$b, 2$}}
    }
    child {node {$\parallel$}
      child {node {$c, 4$}}
      child {node {$e, 3$}
        child {node {$d, 5$}
        }
      }
    };
\end{tikzpicture}
\end{center}
\caption{An increasingly labeled tree with $5$ actions.}
\end{figure}

The formula for computing the number of increasingly labeled processes according to the number of actions was obtained in \cite{Bodini}. We find another formula, and that the number of increasingly labeled processes equals the number of $2\text{-colored}$ chain-increasing binary trees.

\begin{crl}
    The number of increasingly labeled processes with $s$ actions equals the number of $2\text{-colored}$ chain-increasing binary trees with $s$ chains. In other words,

    \[
    y_s(2) = (-1) ^ {s - 1} \sum_{k = 0} ^ {s} (-3) ^ {k} (!(s + k - 1)_k).
    \]
\end{crl}

\begin{proof}
    It is evident that increasingly labeled process trees are equivalent to chain-increasing binary trees. It is because binary vertices in process trees are unlabeled, and unary vertices or leaves are labeled, and the labels increase from the root. In other words, the vertices that correspond to the composition operation are junctions, and the vertices that correspond to actions are chains. However, process trees are planar, which means that the order of branches matters. There are $2$ ways to permute the branches. Allowing permuting branches can be equivalently stated as coloring junctions in $2$ colors. Thus, the number of increasingly labeled processes with $s$ actions is the same as the number of $2\text{-colored}$ chain-increasing binary trees with $s$ chains.

    By Theorem \ref{binary=series-reduced}, we have 
    \[
        y_s(2) = a_s(3) = (-1) ^ {s - 1} \sum_{k = 0} ^ {s} (-3) ^ {k} (!(s + k - 1)_k).
    \]
\end{proof}

We obtain that the number of increasingly labeled processes also equals the number of labeled $3\text{-partite}$ series-reduced trees, which are bijective with symbolic ultrametrics $D : X \times X \rightarrow M$ such that $|M| = 3$. In other words, there may be a connection between symbolic ultrametrics and increasingly labeled processes.

It would be interesting to investigate if there is any connection between $m\text{-colored}$ chain-increasing binary trees and parallel processes for general $m > 2$.

\section{Unlabeled Multipartite Series-reduced Trees}

In this section, we enumerate multipartite series-reduced trees with \emph{unlabeled} leaves. Enumeration of rooted unlabeled series-reduced trees was considered in \cite{RiordanSchroeder}, \cite{RiordanShannon}. For simplicity, we will refer to these trees as series-reduced trees if they are uncolored, and $m\text{-partite}$ series-reduced trees otherwise. Series-reduced trees are also called \emph{homeomorphically irreducible trees} \cite{Harary}.

Let $\bar a_s$ be the number of rooted unlabeled series-reduced trees with $s$ leaves. It was shown in \cite{RiordanShannon} that

\begin{equation}\label{RiordanShannon}
    \bar a_s = \sum_{\{r_j \}} \prod_j \binom{\bar a_j + r_j - 1}{r_j}.
\end{equation}
where $\{ r_j\}$ is all partitions of $n$ with at least two parts.

\noindent{\bf Example.} Below are the first $10$ values of $\bar a_s$ (OEIS \href{https://oeis.org/A000669}{A000669}).
\[
1, 1, 2, 5, 12, 33, 90, 261, 766, 2312, \cdots
\]

We refine (\ref{RiordanShannon}) in the following way: let $\bar a_{s, k}$ denote the number of series-reduced trees with $s$ leaves and $k$ inner vertices. Writing the refinement of (\ref{RiordanShannon}) as $\bar a_s ^ {\text{ref}}(t) = \sum_{k = 0} ^ n \bar a_{s, k} t ^ k$, it is obvious that $\bar a_s(t = 1) = \bar a_s$.

\noindent{\bf Example.} Below are the series-reduced trees with $4$ leaves.

\begin{figure}[htb]
\centering
\begin{center}
    \begin{tabular}{ccccc}
    \begin{tikzpicture}[scale=0.6, grow'=up, level 1/.style={sibling distance=0.9cm}, every node/.style={scale=1}]
      \node[node_white, scale=3.2] (x23) {}
        child {node[node_invisible]{}}
        child {node[node_invisible]{}}
        child {node[node_invisible]{}}
        child {node[node_invisible]{}};
        \node[scale=1.0] at ($(x23)+(2.4em, 0em)$) {$t$};
    \end{tikzpicture}
    &
    \begin{tikzpicture}[scale=0.6, grow'=up, level 1/.style={sibling distance=1.5cm}, level 2/.style={sibling distance=1.2cm}, every node/.style={scale=1}]
      \node[node_white, scale=3.2] (x22) {}
        child {node[node_invisible]{}}
        child {node[node_white, scale=3.2] (x12){}
          child {node[node_invisible]{}}
          child {node[node_invisible]{}}
          child {node[node_invisible]{}}};
        \node[scale=1.0] at ($(x22)+(2.4em, 0em)$) {$t$};
        \node[scale=1.0] at ($(x12)+(2.4em, 0em)$) {$t$};
    \end{tikzpicture}
    &
    \begin{tikzpicture}[scale=0.6, grow'=up, level 1/.style={sibling distance=1.2cm}, level 2/.style={sibling distance=1.2cm}, every node/.style={scale=1}]
      \node[node_white, scale=3.2](x22){}
        child {node[node_invisible]{}}
        child {node[node_invisible]{}}
        child {node[node_white, scale=3.2] (x12){}
          child {node[node_invisible]{}}
          child {node[node_invisible]{}}};
          \node[scale=1.0] at ($(x22)+(2.4em, 0em)$) {$t$};
        \node[scale=1.0] at ($(x12)+(2.4em, 0em)$) {$t$};
    \end{tikzpicture}
    &
    
    \begin{tikzpicture}[scale=0.6, grow'=up, level 1/.style={sibling distance=1.5cm}, level 2/.style={sibling distance=1.2cm}, every node/.style={scale=1}]
      \node[node_white, scale=3.2](x22){}
        child {node[node_invisible]{}}
        child {node[node_white, scale=3.2] (x12){}
          child {node[node_invisible]{}}
          child {node[node_white, scale=3.2] (x32){}
            child {node[node_invisible]{}}
            child {node[node_invisible]{}}}};
      \node[scale=1.0] at ($(x22)+(2.4em, 0em)$) {$t$};
      \node[scale=1.0] at ($(x12)+(2.4em, 0em)$) {$t$};
      \node[scale=1.0] at ($(x32)+(2.4em, 0em)$) {$t$};
    \end{tikzpicture}
    &

    \begin{tikzpicture}[scale=0.6, grow'=up, level 1/.style={sibling distance=1.5cm}, level 2/.style={sibling distance=1.2cm}, every node/.style={scale=1}]
  \node[node_white, scale=3.2] (x22) {}
    child {node[node_white, scale=3.2] (x32){}
      child {node[node_invisible]{}}
      child {node[node_invisible]{}}}
    child {node[node_white, scale=3.2] (x12){}
      child {node[node_invisible]{}}
      child {node[node_invisible]{}}};
  \node[scale=1.0] at ($(x22)+(2.4em, 0em)$) {$t$};
  \node[scale=1.0] at ($(x32)+(-2.4em, 0em)$) {$t$};
  \node[scale=1.0] at ($(x12)+(2.4em, 0em)$) {$t$};
\end{tikzpicture}
    \end{tabular}
\end{center}
\caption{Series-reduced trees with $4$ leaves.}
\end{figure}
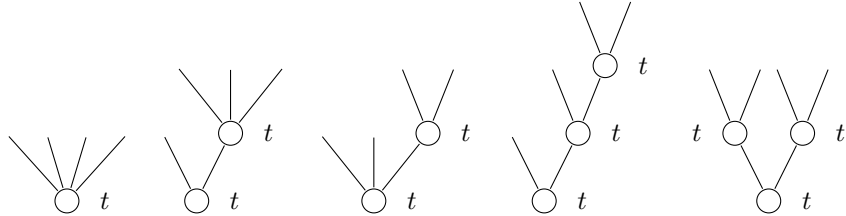
We obtain that $\bar a_4 ^ {\text{ref}}(t) = t + t ^ 2 + t ^ 2 + t ^ 3 + t ^ 3 = t + 2t ^ 2 + 2t ^ 3$. It is obvious $\bar a_1 ^ {\text{ref}}(t) = 1$. For $s > 1$ leaves, we find a recurrence relation to compute the polynomials $\bar a_s ^ {\text{ref}}(t).$

\begin{thm}\label{refinedRiordanShannon}
    Polynomial $\bar a^{\text{ref}}_s(t)$ satisfies the following recurrence relationship for $s > 1$ leaves
    \[
    \bar a_s ^ {\text{ref}}(t) = \frac{t}{s!} \sum_{j = 1} ^ s B_{s, j} \left( \left\{ n! \sum_{\substack{d \mid n \\ n / d \neq s}} \frac{1}{d} \bar a_{n / d}^{\text{ref}}\left(t ^ d \right) \right\}_{n = 1} ^ {s - j + 1} \right),
    \]
    and $\bar a_1 ^ {\text{ref}}(t) = 1.$
\end{thm}

\begin{proof}
    Let $\bar{\mathscr A}_{s}$ denote the set of all non-empty series-reduced trees with no more than $s$ leaves. Also, if $\alpha \in \bar{\mathscr A}_{s}$, denote $l(\alpha)$ as the number of leaves and $i(\alpha)$ as the number of inner vertices in tree $\alpha$. Then, for all $s > 0$, the following holds
    \begin{equation}\label{defa^ref} 
        \bar a_s ^ {\text{ref}}(t) = \sum_{\substack{\alpha \in \bar{\mathscr A}_{s} \\ l(\alpha) = s}} t ^ {i(\alpha)}.
    \end{equation}
    Also, it is evident that
    \begin{equation}\label{raw_unlabeled}
        \bar a_s ^ {\text{ref}}(t) = t [u ^ s] \prod_{\alpha \in \bar{\mathscr A}_{s - 1}} \sum_{k \geq 1} \left(u ^ {l(\alpha)} t ^ {i(\alpha)}\right) ^ k = t [u ^ s] \prod_{\alpha \in \bar{\mathscr A}_{s - 1}} \frac{1}{1 - u ^ {l(\alpha)} t ^ {i(\alpha)}}.
    \end{equation}
    Expression (\ref{raw_unlabeled}) can be thought of as collecting trees with no more than $s - 1$ leaves and attaching them to a common root. The reason we only collect trees with no more than $s - 1$ leaves is that we cannot add a tree with $s$ leaves to the root, because the root would have an out-degree $1$, which is not allowed in series-reduced trees. 
    
    Our goal is to compute the coefficient of $u ^ s$. Let's introduce a temporary variable $r_s(u)$ which is the $\log$ of the product in (\ref{raw_unlabeled})
    \[
    r_s(u) = \log \prod_{\alpha \in \bar{\mathscr A}_{s - 1}} \frac{1}{1 - u ^ {l(\alpha)} t ^ {i(\alpha)}} =
    \sum_{\alpha \in \bar{\mathscr A}_{s - 1}}\log \frac{1}{1 - u ^ {l(\alpha)} t ^ {i(\alpha)}} =
    \]
    \[
    \sum_{\alpha \in \bar{\mathscr A}_{s - 1}} \sum_{d \geq 1} \frac{u ^ {l(\alpha) d} t ^ {i(\alpha) d}}{d} =
    \sum_{d \geq 1} \sum_{\alpha \in \bar{\mathscr A}_{s - 1}} u ^ {l(\alpha) d} \frac{\left(t ^ {i(\alpha)}\right) ^ d}{d}.
    \]
    Take  $n = l(\alpha) d.$ We obtain
    \[
    r_s(u) = \sum_{n \geq 1} u ^ n \sum_{d | n} \frac{1}{d} \sum_{\substack{\alpha \in \bar{\mathscr A}_{s - 1} \\ l(\alpha) = n / d}} \left(t ^ {i(\alpha)}\right) ^ d.
    \]
    Here $d|n$ means $d$ is a divisor of $n$. By (\ref{defa^ref}), we get
    \begin{equation}\label{exact_r_s(u)}
        r_s(u) = \sum_{n \geq 1} \frac{u ^ n}{n!} n! \sum_{\substack{d | n \\ n / d < s}} \frac{1}{d} \bar a ^ {\text{ref}}_{n / d} (t ^ d).
    \end{equation}
    By definition of $r_s(u)$, we know that
    \begin{equation}\label{def_r(u)}
        [u ^ s]e ^{r_s(u)} = [u ^ s] \prod_{\alpha \in \bar{\mathscr A}_{s - 1}} \frac{1}{1 - u ^ {l(\alpha)} t ^ {i(\alpha)}}.
    \end{equation}

    Let us use the series expansion (\ref{exact_r_s(u)}) to evaluate $e ^ {r_s(u)}$ by 
    Bell product
    \[
    e ^ {r_s(u)} = 1 + \sum_{m \geq 1} \frac{r_s(u) ^ m}{m!} = 1 + \sum_{m \geq 1} \frac{u ^ m}{m!} \sum_{j = 1} ^ {m} B_{m, j}\left(\left\{n! \sum_{\substack{d | n \\ n / d <s}} \frac{1}{d} \bar a ^ {\text{ref}}_{n / d} (t ^ d)\right\}_{n = 1} ^ {m - j + 1}\right).
    \]
    By (\ref{def_r(u)}), the expression for the coefficient is
    \[
    [u ^ s] \prod_{\alpha \in \bar{\mathscr A}_{s - 1}} \frac{1}{1 - u ^ {l(\alpha)} t ^ {i(\alpha)}} = \frac{1}{s!}\sum_{j = 1} ^ {s}B_{s, j}\left(\left\{n! \sum_{\substack{d | n \\ n / d < s}} \frac{1}{d} \bar a ^ {\text{ref}}_{n / d} (t ^ d)\right\}_{n = 1} ^ {s - j + 1}\right).
    \]
    Note that the condition $n / d < s$ can be strengthened to $n / d \neq s$ because of $n \leq s - j + 1 \leq s$. Thus, 
    \[
    \bar a_s ^ {\text{ref}}(t) = \frac{t}{s!} \sum_{j = 1} ^ s B_{s, j} \left( \left\{ n! \sum_{\substack{d \mid n \\ n /d \neq s}} \frac{1}{d} \bar a_{n / d}^{\text{ref}}\left(t ^ d \right) \right\}_{n = 1} ^ {s - j + 1} \right).
    \]
\end{proof}

The numerical values of coefficients were obtained by Riordan (\cite{RiordanSchroeder}, p.6)

\begin{table}[h!]
\centering
\begin{tabular}{c|rrrrrrrrr}
$k \backslash n$ & 2 & 3 & 4 & 5 & 6 & 7 & 8 & 9 & 10 \\
\hline
1 & 1 & 1 & 1 & 1 & 1 & 1 & 1 & 1 & 1 \\
2 &   & 1 & 2 & 3 & 4 & 5 & 6 & 7 & 8 \\
3 &   &   & 2 & 5 & 10 & 16 & 24 & 33 & 44 \\
4 &   &   &   & 3 & 12 & 29 & 57 & 99 & 157 \\
5 &   &   &   &   & 6 & 28 & 84 & 192 & 382 \\
6 &   &   &   &   &   & 11 & 66 & 231 & 615 \\
7 &   &   &   &   &   &    & 23 & 157 & 634 \\
8 &   &   &   &   &   &    &    & 46 & 373 \\
9 &   &   &   &   &   &    &    &    & 98 \\
\hline
$\sum$ & 1 & 2 & 5 & 12 & 33 & 90 & 261 & 766 & 2312 \\
\end{tabular}
\caption{Number of rooted series-reduced trees with $s$ leaves and $k$ inner vertices.}
\end{table}

We can now compute the number $\bar a_s(m)$ of multipartite series-reduced trees with $s$ leaves. For that, we use Theorem \ref{refinedRiordanShannon}. Note that we only color inner vertices. 

\begin{thm}\label{mpart-unlabeled}
    The number $\bar a_s(m)$ of multipartite series-reduced trees with $s > 1$ leaves can be expressed as follows
    \[
    \bar a_s(m) = \frac{m}{m - 1}\bar a_s ^ {\text{ref}}(m - 1).
    \]
\end{thm}

\begin{proof}
    Note that the number of ways to color a tree with $k$ inner vertices, such that no two adjacent inner vertices have the same color, is $m (m - 1) ^ {k - 1}$ where $m$ is the number of colors. Thus, it is obvious that the number of $m\text{-partite}$ series-reduced trees is expressed as follows

    \[
    \bar a_s(m) = \bar a_s ^ {\text{ref}}(m - 1) \frac{1}{m - 1} m.
    \]
\end{proof}
\noindent{\bf Example.} Below we provide polynomials $\bar a_s(m)$ for $s = 1, 2, \cdots, 8.$
\[
\bar a_1(m) = 1,
\]
\[
\bar a_2(m) = m,
\]
\[
\bar a_3(m) = m ^ 2,
\]
\[
\bar a_4(m) = m - 2m ^ 2 + 2m ^ 3,
\]
\[
\bar a_5(m) = 2 m^2 - 4 m^3 + 3 m^4,
\]
\[
\bar a_6(m) = m - 4 m^2 + 10 m^3 - 12 m^4 + 6 m^5,
\]
\[
\bar a_7(m) = 3 m^2 - 13 m^3 + 27 m^4 - 27 m^5 + 11 m^6,
\]
\[
\bar a_8(m) = 3 m - 15 m^2 + 42 m^3 - 79 m^4 + 99 m^5 - 72 m^6 + 23 m^7.
\]

The number $\bar a_s(m)$ of $m\text{-partite}$ series-reduced trees with $s$ leaves is given in the table below.

\begin{table}[h!]
\centering
\begin{tabular}{c|rrrrrrrr}
$m \backslash s$ & 1 & 2 & 3 & 4 & 5 & 6 & 7 & 8 \\
\hline
1 & 1 & 1 & 1 & 1 & 1 & 1 & 1 & 1 \\
2 & 1 & 2 & 4 & 10 & 24 & 66 & 180 & 522 \\
3 & 1 & 3 & 9 & 39 & 153 & 723 & 3321 & 16479 \\
4 & 1 & 4 & 16 & 100 & 544 & 3652 & 23536 & 165532 \\
5 & 1 & 5 & 25 & 205 & 1425 & 12405 & 102825 & 936765 \\
6 & 1 & 6 & 36 & 366 & 3096 & 33126 & 335556 & 3755286 \\
7 & 1 & 7 & 49 & 595 & 5929 & 75271 & 900865 & 11958667 \\
8 & 1 & 8 & 64 & 904 & 10368 & 152328 & 2102976 & 32301144 \\
\end{tabular}
\caption{The number of $m\text{-partite}$ series-reduced trees with $s$ leaves.}
\end{table}

One can also obtain the number $\bar a_s ^ {\text{fc}}(m)$ of \emph{fully colored} $m\text{-partite}$ series-reduced trees where inner vertices and leaves are colored.

\begin{thm}
    Number $\bar a_s ^ {\text{fc}}(m)$ of fully colored $m\text{-partite}$ series-reduced trees with $s > 1$ leaves can be written as
    \[
    \bar a_s ^ {\text{fc}}(m) = m (m - 1) ^ {s - 1} \bar a_s ^ {ref}(m - 1).
    \]
\end{thm}

\begin{proof}
    If $s > 1$, there is exactly $(m - 1)$ ways to color a leaf. Then, there is $(m - 1) ^ s$ ways to color an $m\text{-partite}$ series-reduced tree (only inner vertices are colored) with $s$ leaves so it becomes fully colored. Thus, we obtain that
    \[
    \bar a_s ^ {\text{fc}}(m) = (m - 1) ^ {s - 1} \bar a_s(m).
    \]
    By Theorem \ref{mpart-unlabeled}, we get
    \[
    \bar a_s ^ {\text{fc}}(m) = m (m - 1) ^ {s - 1} \bar a_s ^ {ref}(m - 1).
    \]
\end{proof}

\newpage

Below are the numerical values $\bar a^{\text{fc}}_s(m)$ of the fully colored $m\text{-partite}$ series-reduced trees.

\begin{table}[h]
\centering
\begin{tabular}{c|cccccc}
$m \backslash s$ & 1 & 2 & 3 & 4 & 5 & 6 \\
\hline
1 & 1 & 0 & 0 & 0 & 0 & 0 \\
2 & 2 & 2 & 4 & 10 & 24 & 66 \\
3 & 3 & 12 & 72 & 624 & 4896 & 46272 \\
4 & 4 & 36 & 432 & 8100 & 132192 & 2662308 \\
5 & 5 & 80 & 1600 & 52480 & 1459200 & 50810880 \\
6 & 6 & 150 & 4500 & 228750 & 9675000 & 517593750 \\
\end{tabular}
\caption{Number of fully colored $m\text{-partite}$ series-reduced trees with \(s\) leaves.}
\end{table}
Corollary \ref{main_crl} gives the exact form of coefficients in $a_s(m)$ (number of labeled $m\text{-partite}$ series-reduced trees). They are the number of derangements with a fixed number of cycles. It would be interesting to find a nice formula for the coefficients of $\bar a_s(m)$ and give an interpretation.

\section{Acknowledgements}

This research was funded by the  Science Committee of the Ministry of Science and Higher Education of the Republic of Kazakhstan (Grant No. BR 28713025). 


\end{document}